\documentclass[reqno]{amsart}
\usepackage{amssymb,amsmath,hyperref}
\usepackage{amsrefs, float}
\usepackage[foot]{amsaddr}
\usepackage{bbold,stackrel}

\newcommand{\Z}{{\textsf{\textup{Z}}}}

\newtheorem{thm}{Theorem}

\newtheorem{cor}[thm]{Corollary}
\newtheorem{defi}[thm]{Definition}
\newtheorem{rem}[thm]{Remark}
\newtheorem{nota}[thm]{Notation}

\newtheorem{princ}[thm]{Principle}

\newtheorem{ack}[thm]{Acknowledgement}

\newtheorem*{tempo*}{Template}

\newcommand\be{\begin{equation}}
\newcommand\ee{\end{equation}} 

\usepackage{amsmath,amsfonts} 

\usepackage[applemac]{inputenc}

\def\bdefi{\begin{defi}\rm}
\def\edefi{\end{defi}}
\def\bnota{\begin{nota}\rm}
\def\enota{\end{nota}}
\def\FIVE{\Pi_{1}^{1}\text{-\textup{\textsf{CA}}}_{0}}
\def\SIX{\Pi_{2}^{1}\text{-\textsf{\textup{CA}}}_{0}}

\def\SIXK{\Pi_{k}^{1}\text{-\textsf{\textup{CA}}}_{0}^{\omega}}

\def\ZF{\textup{\textsf{ZF}}}

\def\L{\textsf{\textup{L}}}

 \def\r{\mathbb{r}}

\def\RCA{\textup{\textsf{RCA}}}
\def\({\textup{(}}
\def\){\textup{)}}

\def\RCAo{\textup{\textsf{RCA}}_{0}^{\omega}}
\def\ACAo{\textup{\textsf{ACA}}_{0}^{\omega}}

\def\bye{\end{document}}
\def\N{{\mathbb  N}}
\def\Q{{\mathbb  Q}}
\def\R{{\mathbb  R}}

\def\SS{\textup{\textsf{S}}}

\def\di{\rightarrow}

\def\asa{\leftrightarrow}

\def\ACA{\textup{\textsf{ACA}}}

\def\QFAC{\textup{\textsf{QF-AC}}}

\def\STS{\textup{\textsf{STS}}}

\def\cocode{\textup{\textsf{cocode}}}
\def\SUM{\textup{\textsf{SUM}}}
\def\BAIRE{\textup{\textsf{Baire}}}
\def\NIN{\textup{\textsf{NIN}}}

\def\NBI{\textup{\textsf{NBI}}}

\def\BOOT{\textup{\textsf{BOOT}}}

\def\IND{\textup{\textsf{IND}}}
\def\NFP{\textup{\textsf{NFP}}}

\def\fin{\textup{\textsf{fin}}}

\def\eps{\varepsilon}

\def\ECF{\textup{\textsf{ECF}}}

\usepackage{graphicx}
\usepackage{tikz}
\usetikzlibrary{matrix, shapes.misc}

\numberwithin{equation}{section}
\numberwithin{thm}{section}

\usepackage{comment,tikz-cd}

\begin{document}
\title[Reverse Mathematics of the uncountability of $\R$]{Reverse Mathematics of the uncountability of $\R$: \\
{\tiny Baire classes, metric spaces, and unordered sums}}
\author{Sam Sanders}
\address{Department of Mathematics, TU Darmstadt, Darmstadt, Germany}
\email{sasander@me.com, sanders@mathematik.tu-darmstadt.de}
\begin{abstract}
Dag Normann and the author have recently initiated the study of the logical and computational properties of \emph{the uncountability of $\R$} formalised as the statement $\NIN$ (resp.\ $\NBI$) 
that \emph{there is no injection \(resp.\ bijection\) from $[0,1]$ to $\N$}.  On one hand, these principles are \emph{hard to prove} relative to the usual scale based on comprehension and discontinuous functionals. 
On the other hand, these principles are among the weakest principles on a new \emph{complimentary} scale based on (classically valid) continuity axioms from Brouwer's intuitionistic mathematics. 
We continue the study of $\NIN$ and $\NBI$ relative to the latter scale, connecting these principles with theorems about \emph{Baire classes}, \emph{metric spaces}, and \emph{unordered sums}.  The importance of the first two topics requires no explanation, while the final topic's main theorem, i.e.\ that \emph{when they exist, unordered sums are \(countable) series}, has the rather unique property 
of implying $\NIN$ formulated with the Cauchy criterion, and (only) $\NBI$ when formulated with limits. 
%the main theorem concerning the latter implies 
This study is undertaken within Ulrich Kohlenbach's framework of \emph{higher-order Reverse Mathematics}.
\end{abstract}
%
%\setcounter{page}{0}
%\tableofcontents
%\thispagestyle{empty}
%\newpage

\maketitle
\thispagestyle{empty}

\section{Introduction}\label{intro}
The uncountability of $\R$ deals with arbitrary mappings with domain $\R$, and is therefore best studied in a language that has such objects as first-class citizens. 
Obviousness, much more than beauty, is however in the eye of the beholder.  Lest we be misunderstood, we formulate a blanket caveat: all notions (computation, continuity, function, open set, et cetera) used in this paper are to be interpreted via their higher-order definitions, also listed below, \emph{unless explicitly stated otherwise}. 
\subsection{Aim and motivation}\label{sintro}
In a nutshell, we study the \emph{Reverse Mathematics} of the \emph{uncountability of $\R$}, in particular the connection between the latter and \emph{Baire classes}, \emph{unordered sums}, and \emph{metric spaces}.
We now explain the first two italicised concepts, as follows.  

\smallskip

First of all, the uncountability of $\R$ is perhaps the most basic property of the real numbers, to be found in many mainstream textbooks.  
The uncountability of $\R$ can be expressed in at least three different ways:
\begin{enumerate}
 \renewcommand{\theenumi}{\alph{enumi}}
\item \emph{Cantor's theorem}: there is no surjection from $\N$ to $\R$.\label{dak}
\item There is no injection from $\R$ to $\N$ (Kunen, \cite{kunen})\label{ku}
\item There is no bijjection from $\R$ to $\N$ (Hbracek-Jech, \cite{hrbacekjech})\label{hj}
\end{enumerate}
Cantor established the uncountability of $\R$ in 1874 in his \emph{first} set theory paper \cite{cantor1} based on item \eqref{dak} formulated as in the following theorem. 
\begin{thm}[Cantor's theorem]\label{lanto}
\emph{For a sequence of real numbers, there is a real number not in that sequence.}
\end{thm}
Secondly, the program \emph{Reverse Mathematics} seeks to identify the minimal axioms needed to prove theorems of ordinary\footnote{Simpson describes \emph{ordinary mathematics} in \cite{simpson2}*{I.1} as \emph{that body of mathematics that is prior to or independent of the introduction of abstract set theoretic concepts}.} mathematics.  We provide an introduction to Reverse Mathematics (RM hereafter) in Section \ref{prelim1}.
The uncountability of $\R$ is studied in (second-order) RM in \cite{simpson2}*{II.4.9} in the form of Theorem~\ref{lanto}.  

\smallskip

Thirdly, since Cantor's theorem as in item \eqref{dak} and Theorem \ref{lanto} is studied in RM, it is a natural question what the status is of the remaining items \eqref{ku} and \eqref{hj}, working in Kohlenbach's \emph{higher-order} RM, where the latter is introduced in Section~\ref{prelim1}.
Hence, Dag Normann and the author initiated the study the uncountability of $\R$ in \cite{dagsamX} in the guise of the following natural principles:
\begin{itemize}
\item $\NIN$: \emph{there is no injection from $[0,1]$ to $\N$},
\item  $\NBI$: \emph{there is no bijection from $[0,1]$ to $\N$}.
\end{itemize}
Now Cantor's theorem as in item \eqref{dak} is provable in the base theory of RM, and therefore classified as `weak'.  % The results in \cites{dagsamX} 
%In this paper, \emph{principle} generally refers to a statement of \emph{ordinary mathematics}\footnote{Simpson describes \emph{ordinary mathematics} in \cite{simpson2}*{I.1} as \emph{that body of mathematics that is prior to or independent of the introduction of abstract set theoretic concepts}.  The uncountability of $\R$ is studied by Simpson in \cite{simpson2}*{II.4.9}, i.e.\ the former seems to count as ordinary.}, and our aim is to investigate the logical and computational properties of these.
%The principle $\NIN$ will take centre stage, while $\NBI$ will be shown to have some interesting properties as well. 
%Now, a central and important aspect of mathematical logic is the classification of principles and objects in hierarchies based on logical or computational strength. 
%A natural question would therefore \emph{seem to be} where $\NIN$ is located in the well-known hierarchies of logical and computational strength, generally based on \emph{comprehension} and \emph{discontinuous} functionals.  
Intuitively speaking, $\NIN$ and $\NBI$ are also weak principles, yet we need rather strong comprehension axioms to prove them, namely at the level of second-order arithmetic, by \cite{dagsamX}*{Theorem 3.1}.  
Many theorems with this property (intuitively weak, but requiring strong comprehension axioms for a proof) have been identified in \cites{dagsamIII, dagsamV, dagsamVI, dagsamVII, dagsamX}, suggesting the need for an 
alternative scale that classifies `intuitively weak' theorems as `formally weak'.  

\smallskip

% and
%the reason for this paradox is that we are comparing two fundamentally different classes. 
%Indeed, a fundamental division here is between \emph{normal} and \emph{non-normal} objects and principles, where the former give rise to \emph{discontinuous objects} and the latter do not (see Definition \ref{norma} for the exact formulation).  
%For reference, $\NIN$ and $\NBI$ are \emph{non-normal} as they do not imply the existence of a discontinuous function on $\R$.
%
%\smallskip

Finally, we have developed such an alternative scale in \cites{dagsamX, samNEO, samph}, based on \emph{classically valid} continuity axioms from Brouwer's intuitionistic mathematics, namely the \emph{neighbourhood function principle} $\NFP$ from \cite{troeleke1}.  As is clear from \cite{dagsamX}*{Figure~1}, $\NIN$ and $\NBI$ are among the weakest principles on this alternative scale.  
In this paper, we classify the following theorems relative to this scale by connecting them to $\NIN$ and $\NBI$; the topics are \emph{unordered sums}, \emph{metric spaces}, and \emph{Baire classes}.  %, and principles related to countable sets.  % and related principles concerning countable sets, like $\cocode_{0}$ from Section \ref{pintro}.
\begin{itemize}
\item There is a function not in Baire class $2$.
\item The characterisation theorem for Baire class $1$.
\item Locally, sequential continuity implies continuity in metric spaces. 
\item A metric space is separable if it is countably compact. 
\item An unordered sum, if it exists, equals a sub-series.  
\end{itemize}
We also connect these items to basic theorems about countable sets, as such theorems can be `explosive' in that they become much stronger when combined with discontinuous comprehension functionals. 
The aforementioned `alternative scale' based on $\NFP$ is discussed in detail in Section \ref{kelim}, up next.

%Mention (A) here: countable set is not all of $[0,1]$.

\subsection{An alternative scale following Brouwer}\label{kelim}
We discuss the alternative scale based on $\NFP$ mentioned in the previous section.  
The systems $\Z_{2}^{\omega}$ and $\Z_{2}^{\Omega}$ from Section \ref{HCT} play a central role.  
For now, it suffices to know that both these systems are conservative extensions of second-order arithmetic $\Z_{2}$, while
$\Z_{2}^{\omega}$ (resp.\ $\Z_{2}^{\Omega}$) is based on third-order (resp.\ fourth-order) comprehension functionals.  

\smallskip

Now, the following theorems formulated in the language of third-order arithmetic, are not provable in $\Z_{2}^{\omega}$, but provable in $\Z_{2}^{\Omega}$, by the results in \cites{dagsamIII, dagsamV, dagsamVI, dagsamX}.  
\begin{itemize}
\item Arzel\`a's convergence theorem for the Riemann integral (1885, \cites{arse2}).
\item Pincherle's local-global theorem (1882, \cite{tepelpinch}).
\item The uncountability of $\R$ as in $\NIN$ or $\NBI$ (1874, Cantor, \cite{cantor1}).
\item Covering theorems (Lindel\"of, Heine-Borel, Vitali, Besicovitch, \dots) for uncountable coverings of the unit interval. 
\item Basic properties of the gauge and Lebesgue integral (\cite{zwette}) without the coding from in second-order RM (see \cite{simpson2}*{X.1}).
\item Basic theorems (Urysohn, Tietze, \dots) concerning open sets given as characteristic functions (\cite{dagsamVII}).
\item Covering theorems (Heine-Borel, Vitali) for coverings of the unit interval formulated with \textbf{countable collections}, as done by Borel in \cite{opborrelen2}.
\item Convergence theorems for nets in the unit interval indexed by Baire space. 
\item Basic theorems on countable sets (defined via injections or bijections to $\N$ as in Definition \ref{standard}), like that \emph{a countable set has measure zero}.  
\end{itemize}
This list may be greatly extended by consulting \cites{dagsamIII, dagsamV, dagsamVI, dagsamX}.  The point of this list is to exhibit a large number of \emph{intuitively weak} theorems of ordinary mathematics that are formally classified as \emph{hard to prove}; indeed, $\Z_{2}^{\omega}$ does not suffice for a proof of any of the above items, while $\Z_{2}^{\Omega}$ does, and these two systems are both conservative extensions of $\Z_{2}$.  Hence, we observe a conceptual problem with the classification of third-order principles based on third-order comprehension functionals as in $\Z_{2}^{\omega}$. 

\smallskip

The cause of the above problem is that we are mixing two fundamentally different categories.
Indeed, a fundamental division here is between \emph{normal} and \emph{non-normal} objects and principles, where the former give rise to \emph{discontinuous objects} and the latter do not (see Definition \ref{norma} for the exact formulation).  
For reference, $\NIN$ and $\NBI$ are \emph{non-normal} as they do not imply the existence of a discontinuous function on $\R$.
In this paper, all principles we study are part of third-order arithmetic, i.e.\ `non-normal vs normal' refers to the existence of a discontinuous function on $\R$.
%The associated computations are one type-level higher. 
% following Footnote \ref{frakker}.   

\smallskip

Now, the `normal vs non-normal' distinction yield two (fairly independent) scales for classifying logical and computational strength: the standard one is the `normal' scale based on comprehension and \emph{discontinuous} objects, like the G\"odel hierarchy (\cite{sigohi}) and \emph{higher-order} Reverse Mathematics (Section \ref{prelim1}).  The `non-normal' scale is a hierarchy based on the \emph{neighbourhood function principle} $\NFP$ from \cite{troeleke1}*{p.\ 215}, a classically valid continuity axiom of Brouwer's intuitionistic mathematics. 
\begin{princ}[$\NFP$]
For any formula $A$, 
\be\label{gru}
(\forall f \in \N^{\N})(\exists n \in \N)A(\overline{f}n)\di  (\exists g\in K_{0})(\forall f\in \N^{\N})A(\overline{f}g(f)),
\ee   
where `$g\in K_{0}$' means that $g$ is an RM-code and $\overline{f}n$ is $\langle f(0), f(1), \dots, f(n-1) \rangle$.
\end{princ}
A classification of convergence theorems for nets and uncountable covering theorems in terms of $\NFP$ can be found in \cite{samph}, while the connection to $\NIN$ and $\NBI$ is explored in \cite{dagsamX, samNEO}.
As is clear from \cite{dagsamX}*{Figure 1}, $\NIN$ and $\NBI$ are among the weakest principles on the non-normal scale. 

\smallskip

Finally, applying the canonical embedding of higher-order arithmetic to second-order arithmetic, called $\ECF$ in \cite{kohlenbach2}, many of the results in \cite{samph} yield known results regarding the Big Five of RM.  Hence, second-order RM is a reflection of a higher truth under a lossy translation (namely $\ECF$), following Plato's \emph{allegory of the cave}.  We discuss $\ECF$ in more detail in Remark \ref{ECF}.

\section{Preliminaries}
We introduce \emph{Reverse Mathematics} in Section \ref{prelim1}, as well as Kohlebach's generalisation to \emph{higher-order arithmetic}, and the associated base theory $\RCAo$.  % in Section \ref{KOH}.  
We introduce some notations in Section \ref{kota} and higher-order axioms in Section~\ref{HCT}.  % while Section \ref{kelim} introduces the alternative scale mentioned in Section \ref{intro}.

\subsection{Reverse Mathematics}\label{prelim1}
Reverse Mathematics is a program in the foundations of mathematics initiated around 1975 by Friedman (\cites{fried,fried2}) and developed extensively by Simpson (\cite{simpson2}).  
The aim of RM is to identify the minimal axioms needed to prove theorems of ordinary, i.e.\ non-set theoretical, mathematics. 

\smallskip

We refer to \cite{stillebron} for a basic introduction to RM and to \cite{simpson2, simpson1} for an overview of RM.  We expect familiarity with RM, but do sketch some aspects of Kohlenbach's \emph{higher-order} RM (\cite{kohlenbach2}) essential to this paper, including the base theory $\RCAo$ (Definition \ref{kase}).  
As will become clear, the latter is officially a type theory but can accommodate (enough) set theory via e.g.\ Definition \ref{keepintireal}.\eqref{koer} and Definition \ref{standard}. 

\smallskip

First of all, in contrast to `classical' RM based on \emph{second-order arithmetic} $\Z_{2}$, higher-order RM uses $\L_{\omega}$, the richer language of \emph{higher-order arithmetic}.  
Indeed, while the former is restricted to natural numbers and sets of natural numbers, higher-order arithmetic can accommodate sets of sets of natural numbers, sets of sets of sets of natural numbers, et cetera.  
To formalise this idea, we introduce the collection of \emph{all finite types} $\mathbf{T}$, defined by the two clauses:
\begin{center}
(i) $0\in \mathbf{T}$   and   (ii)  If $\sigma, \tau\in \mathbf{T}$ then $( \sigma \di \tau) \in \mathbf{T}$,
\end{center}
where $0$ is the type of natural numbers, and $\sigma\di \tau$ is the type of mappings from objects of type $\sigma$ to objects of type $\tau$.
In this way, $1\equiv 0\di 0$ is the type of functions from numbers to numbers, and  $n+1\equiv n\di 0$.  Viewing sets as given by characteristic functions, we note that $\Z_{2}$ only includes objects of type $0$ and $1$.    

\smallskip

Secondly, the language $\L_{\omega}$ includes variables $x^{\rho}, y^{\rho}, z^{\rho},\dots$ of any finite type $\rho\in \mathbf{T}$.  Types may be omitted when they can be inferred from context.  
The constants of $\L_{\omega}$ include the type $0$ objects $0, 1$ and $ <_{0}, +_{0}, \times_{0},=_{0}$  which are intended to have their usual meaning as operations on $\N$.
Equality at higher types is defined in terms of `$=_{0}$' as follows: for any objects $x^{\tau}, y^{\tau}$, we have
\be\label{aparth}
[x=_{\tau}y] \equiv (\forall z_{1}^{\tau_{1}}\dots z_{k}^{\tau_{k}})[xz_{1}\dots z_{k}=_{0}yz_{1}\dots z_{k}],
\ee
if the type $\tau$ is composed as $\tau\equiv(\tau_{1}\di \dots\di \tau_{k}\di 0)$.  
Furthermore, $\L_{\omega}$ also includes the \emph{recursor constant} $\mathbf{R}_{\sigma}$ for any $\sigma\in \mathbf{T}$, which allows for iteration on type $\sigma$-objects as in the special case \eqref{special}.  Formulas and terms are defined as usual.  
One obtains the sub-language $\L_{n+2}$ by restricting the above type formation rule to produce only type $n+1$ objects (and related types of similar complexity).        
\bdefi\label{kase} 
The base theory $\RCAo$ consists of the following axioms.
\begin{enumerate}
 \renewcommand{\theenumi}{\alph{enumi}}
\item  Basic axioms expressing that $0, 1, <_{0}, +_{0}, \times_{0}$ form an ordered semi-ring with equality $=_{0}$.
\item Basic axioms defining the well-known $\Pi$ and $\Sigma$ combinators (aka $K$ and $S$ in \cite{avi2}), which allow for the definition of \emph{$\lambda$-abstraction}. 
\item The defining axiom of the recursor constant $\mathbf{R}_{0}$: for $m^{0}$ and $f^{1}$: 
\be\label{special}
\mathbf{R}_{0}(f, m, 0):= m \textup{ and } \mathbf{R}_{0}(f, m, n+1):= f(n, \mathbf{R}_{0}(f, m, n)).
\ee
\item The \emph{axiom of extensionality}: for all $\rho, \tau\in \mathbf{T}$, we have:
\be\label{EXT}\tag{$\textsf{\textup{E}}_{\rho, \tau}$}  
(\forall  x^{\rho},y^{\rho}, \varphi^{\rho\di \tau}) \big[x=_{\rho} y \di \varphi(x)=_{\tau}\varphi(y)   \big].
\ee 
\item The induction axiom for quantifier-free formulas of $\L_{\omega}$.
\item $\QFAC^{1,0}$: the quantifier-free Axiom of Choice as in Definition \ref{QFAC}.
\end{enumerate}
\edefi
\noindent
Note that variables (of any finite type) are allowed in quantifier-free formulas of the language $\L_{\omega}$: only quantifiers are banned.
%Moreover, we let $\IND^{\omega}$ be the induction axiom for all formulas in $\L_{\omega}$.
\bdefi\label{QFAC} The axiom $\QFAC$ consists of the following for all $\sigma, \tau \in \textbf{T}$:
\be\tag{$\QFAC^{\sigma,\tau}$}
(\forall x^{\sigma})(\exists y^{\tau})A(x, y)\di (\exists Y^{\sigma\di \tau})(\forall x^{\sigma})A(x, Y(x)),
\ee
for any quantifier-free formula $A$ in the language of $\L_{\omega}$.
\edefi
As discussed in \cite{kohlenbach2}*{\S2}, $\RCAo$ and $\RCA_{0}$ prove the same sentences `up to language' as the latter is set-based and the former function-based.  Recursion as in \eqref{special} is called \emph{primitive recursion}; the class of functionals obtained from $\mathbf{R}_{\rho}$ for all $\rho \in \mathbf{T}$ is called \emph{G\"odel's system $T$} of all (higher-order) primitive recursive functionals.  

\subsection{Notations and the like}\label{kota}
We introduce some relevant notations and basic definitions related to higher-order RM. 

\smallskip

First of all, we use the usual notations for natural, rational, and real numbers, and the associated functions, as introduced in \cite{kohlenbach2}*{p.\ 288-289}.  
\begin{defi}[Real numbers and related notions in $\RCAo$]\label{keepintireal}\rm~
\begin{enumerate}
 \renewcommand{\theenumi}{\alph{enumi}}
\item Natural numbers correspond to type zero objects, and we use `$n^{0}$' and `$n\in \N$' interchangeably.  Rational numbers are defined as signed quotients of natural numbers, and `$q\in \Q$' and `$<_{\Q}$' have their usual meaning.    
\item Real numbers are coded by fast-converging Cauchy sequences $q_{(\cdot)}:\N\di \Q$, i.e.\  such that $(\forall n^{0}, i^{0})(|q_{n}-q_{n+i}|<_{\Q} \frac{1}{2^{n}})$.  
We use Kohlenbach's `hat function' from \cite{kohlenbach2}*{p.\ 289} to guarantee that every $q^{1}$ defines a real number.  
\item We write `$x\in \R$' to express that $x^{1}:=(q^{1}_{(\cdot)})$ represents a real as in the previous item and write $[x](k):=q_{k}$ for the $k$-th approximation of $x$.    
\item Two reals $x, y$ represented by $q_{(\cdot)}$ and $r_{(\cdot)}$ are \emph{equal}, denoted $x=_{\R}y$, if $(\forall n^{0})(|q_{n}-r_{n}|\leq {2^{-n+1}})$. Inequality `$<_{\R}$' is defined similarly.  
We sometimes omit the subscript `$\R$' if it is clear from context.           
\item Functions $F:\R\di \R$ are represented by $\Phi^{1\di 1}$ mapping equal reals to equal reals, i.e.\ extensionality as in $(\forall x , y\in \R)(x=_{\R}y\di \Phi(x)=_{\R}\Phi(y))$.\label{EXTEN}
\item The relation `$x\leq_{\tau}y$' is defined as in \eqref{aparth} but with `$\leq_{0}$' instead of `$=_{0}$'.  Binary sequences are denoted `$f^{1}, g^{1}\leq_{1}1$', but also `$f,g\in C$' or `$f, g\in 2^{\N}$'.  Elements of Baire space are given by $f^{1}, g^{1}$, but also denoted `$f, g\in \N^{\N}$'.
\item For a binary sequence $f^{1}$, the associated real in $[0,1]$ is $\r(f):=\sum_{n=0}^{\infty}\frac{f(n)}{2^{n+1}}$.\label{detrippe}
%\item An object $\textbf{Y}^{0\di \rho}$ is called \emph{a sequence of type $\rho$ objects} and also denoted $\textbf{Y}=(Y_{n})_{n\in \N}$ or $\textbf{Y}=\lambda n. Y_{n}$ where $Y_{n}:=\textbf{Y}(n)$ for all $n^{0}$.
\item Sets of type $\rho$ objects $X^{\rho\di 0}, Y^{\rho\di 0}, \dots$ are given by their characteristic functions $F^{\rho\di 0}_{X}\leq_{\rho\di 0}1$, i.e.\ we write `$x\in X$' for $ F_{X}(x)=_{0}1$. \label{koer} 
\end{enumerate}
\end{defi}
\noindent
We note that sets as in item \eqref{koer} from Definition \ref{keepintireal} are also used in e.g.\ \cite{kruisje, dagsamX}. 

\smallskip

Secondly, we mention the highly useful $\ECF$-interpretation. 
\begin{rem}[The $\ECF$-interpretation]\label{ECF}\rm
The (rather) technical definition of $\ECF$ may be found in \cite{troelstra1}*{p.\ 138, \S2.6}.
Intuitively, the $\ECF$-interpretation $[A]_{\ECF}$ of a formula $A\in \L_{\omega}$ is just $A$ with all variables 
of type two and higher replaced by type one variables ranging over so-called `associates' or `RM-codes' (see \cite{kohlenbach4}*{\S4}); the latter are (countable) representations of continuous functionals.  
The $\ECF$-interpretation connects $\RCAo$ and $\RCA_{0}$ (see \cite{kohlenbach2}*{Prop.\ 3.1}) in that if $\RCAo$ proves $A$, then $\RCA_{0}$ proves $[A]_{\ECF}$, again `up to language', as $\RCA_{0}$ is 
formulated using sets, and $[A]_{\ECF}$ is formulated using types, i.e.\ using type zero and one objects.  
\end{rem}
In light of the widespread use of codes in RM and the common practise of identifying codes with the objects being coded, it is no exaggeration to refer to $\ECF$ as the \emph{canonical} embedding of higher-order into second-order arithmetic. 

\smallskip

Finally, for completeness, we list a notational convention for finite sequences.  
\begin{nota}[Finite sequences]\label{skim}\rm
We assume a dedicated type for `finite sequences of objects of type $\rho$', namely $\rho^{*}$, which we shall only use for $\rho=0,1$.  Since the usual coding of pairs of numbers goes through in $\RCAo$, we shall not always distinguish between $0$ and $0^{*}$. 
Similarly, we do not always distinguish between `$s^{\rho}$' and `$\langle s^{\rho}\rangle$', where the former is `the object $s$ of type $\rho$', and the latter is `the sequence of type $\rho^{*}$ with only element $s^{\rho}$'.  The empty sequence for the type $\rho^{*}$ is denoted by `$\langle \rangle_{\rho}$', usually with the typing omitted.  

\smallskip

Furthermore, we denote by `$|s|=n$' the length of the finite sequence $s^{\rho^{*}}=\langle s_{0}^{\rho},s_{1}^{\rho},\dots,s_{n-1}^{\rho}\rangle$, where $|\langle\rangle|=0$, i.e.\ the empty sequence has length zero.
For sequences $s^{\rho^{*}}, t^{\rho^{*}}$, we denote by `$s*t$' the concatenation of $s$ and $t$, i.e.\ $(s*t)(i)=s(i)$ for $i<|s|$ and $(s*t)(j)=t(|s|-j)$ for $|s|\leq j< |s|+|t|$. For a sequence $s^{\rho^{*}}$, we define $\overline{s}N:=\langle s(0), s(1), \dots,  s(N-1)\rangle $ for $N^{0}<|s|$.  
For a sequence $\alpha^{0\di \rho}$, we also write $\overline{\alpha}N=\langle \alpha(0), \alpha(1),\dots, \alpha(N-1)\rangle$ for \emph{any} $N^{0}$.  By way of shorthand, 
$(\forall q^{\rho}\in Q^{\rho^{*}})A(q)$ abbreviates $(\forall i^{0}<|Q|)A(Q(i))$, which is (equivalent to) quantifier-free if $A$ is.   
\end{nota}

\subsection{Some axioms of higher-order Reverse Mathematics}\label{HCT}
We introduce some axioms of higher-order RM which will be used below.
In particular, we introduce some functionals which constitute the counterparts of second-order arithmetic $\Z_{2}$, and some of the Big Five systems, in higher-order RM.
We use the `standard' formulation from \cite{kohlenbach2, dagsamIII}.  

\smallskip
\noindent
First of all, $\ACA_{0}$ is readily derived from the following sentence:
\begin{align}\label{mu}\tag{$\mu^{2}$}
(\exists \mu^{2})(\forall f^{1})\big[ (\exists n)(f(n)=0) \di [(f(\mu(f))=0)&\wedge (\forall i<\mu(f))f(i)\ne 0 ]\\
& \wedge [ (\forall n)(f(n)\ne0)\di   \mu(f)=0]    \big], \notag
\end{align}
and $\ACA_{0}^{\omega}\equiv\RCAo+(\mu^{2})$ proves the same sentences as $\ACA_{0}$ by \cite{hunterphd}*{Theorem~2.5}.   The (unique) functional $\mu^{2}$ in $(\mu^{2})$ is also called \emph{Feferman's $\mu$} (\cite{avi2}), 
and is clearly \emph{discontinuous} at $f=_{1}11\dots$; in fact, $(\mu^{2})$ is equivalent to the existence of $F:\R\di\R$ such that $F(x)=1$ if $x>_{\R}0$, and $0$ otherwise (\cite{kohlenbach2}*{\S3}), and to 
\be\label{muk}\tag{$\exists^{2}$}
(\exists \varphi^{2}\leq_{2}1)(\forall f^{1})\big[(\exists n)(f(n)=0) \asa \varphi(f)=0    \big]. 
\ee
\noindent
Secondly, $\FIVE$ is readily derived from the following sentence:
\be\tag{$\SS^{2}$}
(\exists\SS^{2}\leq_{2}1)(\forall f^{1})\big[  (\exists g^{1})(\forall n^{0})(f(\overline{g}n)=0)\asa \SS(f)=0  \big], 
\ee
and $\FIVE^{\omega}\equiv \RCAo+(\SS^{2})$ proves the same $\Pi_{3}^{1}$-sentences as $\FIVE$ by \cite{yamayamaharehare}*{Theorem 2.2}.   The (unique) functional $\SS^{2}$ in $(\SS^{2})$ is also called \emph{the Suslin functional} (\cite{kohlenbach2}).
By definition, the Suslin functional $\SS^{2}$ can decide whether a $\Sigma_{1}^{1}$-formula as in the left-hand side of $(\SS^{2})$ is true or false.  

\smallskip

We similarly define the functional $\SS_{k}^{2}$ which decides the truth or falsity of $\Sigma_{k}^{1}$-formulas; we also define 
the system $\SIXK$ as $\RCAo+(\SS_{k}^{2})$, where  $(\SS_{k}^{2})$ expresses that $\SS_{k}^{2}$ exists.  Note that we allow formulas with \emph{function} parameters, but \textbf{not} \emph{functionals} here.
In fact, Gandy's \emph{Superjump} (\cite{supergandy}) constitutes a way of extending $\FIVE^{\omega}$ to parameters of type two.  
We identify the functionals $\exists^{2}$ and $\SS_{0}^{2}$ and the systems $\ACAo$ and $\SIXK$ for $k=0$.
We note that the operators $\nu_{n}$ from \cite{boekskeopendoen}*{p.\ 129} are essentially $\SS_{n}^{2}$ strengthened to return a witness (if existant) to the $\Sigma_{k}^{1}$-formula at hand.  %  if it exists. 

\smallskip

\noindent
Thirdly, full second-order arithmetic $\Z_{2}$ is readily derived from $\cup_{k}\SIXK$, or from:
\be\tag{$\exists^{3}$}
(\exists E^{3}\leq_{3}1)(\forall Y^{2})\big[  (\exists f^{1})(Y(f)=0)\asa E(Y)=0  \big], 
\ee
and we therefore define $\Z_{2}^{\Omega}\equiv \RCAo+(\exists^{3})$ and $\Z_{2}^\omega\equiv \cup_{k}\SIXK$, which are conservative over $\Z_{2}$ by \cite{hunterphd}*{Cor.\ 2.6}. 
Despite this close connection, $\Z_{2}^{\omega}$ and $\Z_{2}^{\Omega}$ can behave quite differently, as discussed in e.g.\ \cite{dagsamIII}*{\S2.2}.   The functional from $(\exists^{3})$ is also called `$\exists^{3}$', and we use the same convention for other functionals.  

\smallskip

Finally, we mention the distinction between `normal' and `non-normal' functionals  based on the following definition from \cite{longmann}*{\S5.4}.  Kleene's computation schemes S1-S9 may be found in \cite{kleeneS1S9, longmann}.
In this paper, we only study statements in the language of third-order arithmetic, i.e.\ we only need Definition \ref{norma} for $n=2$.
\begin{defi}\label{norma}
For $n\geq 2$, a functional of type $n$ is called \emph{normal} if it computes Kleene's $\exists^{n}$ following S1-S9, and \emph{non-normal} otherwise.  
\end{defi}
Similarly, we call a statement about type $n$ objects ($n\geq 2$) \emph{normal} if it implies the existence of $\exists^{n}$ over $\RCAo$ from Section \ref{prelim1}, and \emph{non-normal} otherwise.  
%We also use `\emph{strongly} non-normal' for type $3$ functionals that do not compute $\exists^{3}$ \emph{relative to $\exists^{2}$}.  

\section{Main results}
\subsection{Introduction}\label{pintro}
The uncountability of $\R$ as in $\NIN$ and $\NBI$ follows from numerous basic theorems of ordinary mathematics, as established in \cite{dagsamX} and noted in Section \ref{sintro}. 
In this section, we derive $\NIN$ and/or $\NBI$ from basic theorems pertaining to \emph{Baire classes} (Section~\ref{afbaire}), \emph{unordered sums} (Section~\ref{uncountsyn}), and \emph{metric spaces} (Section \ref{donk}).
We also connect the latter to basic theorems about countable sets, as such theorems are `explosive' in that they can become much stronger, e.g.\ yielding $\SIX$ when combined with $\FIVE^{\omega}$. 

\smallskip

As it happens, $\NIN$ is one of the weakest theorems pertaining to countable sets:  it is namely equivalent to the following centred statement by \cite{samNEO}*{Theorem 3.1}:
\be\tag{\textsf{A}}\label{hong}
\textup{\emph{for countable $A\subset \R$, there is a real $y\in [0,1]$ different from all reals in $A$},}
\ee
which uses the usual definition of \emph{countable set} from Definition \ref{standard}.
By Definition~\ref{keepintireal}.\eqref{koer}, sets $A\subset \R$ are characteristic\footnote{When relevant, we assume $(\exists^{2})$ to make sure that the definition of `open set' from \cite{dagsamVII} also represents a characteristic function.} functions, as in e.g.\ \cites{kruisje, dagsamVI, dagsamVII, hunterphd, dagsamX}.     
 % while the usual definition of `countable set' is used, namely as follows. 
\bdefi[Countable subset of $\R$]\label{standard}~
A set $A\subseteq \R$ is \emph{countable} if there exists $Y:\R\di \N$ such that $(\forall x, y\in A)(Y(x)=_{0}Y(y)\di x=_{\R}y)$. 
\edefi
This definition is from Kunen's textbook on set theory (\cite{kunen}*{p.\ 63}); we \emph{could} additionally require that $Y:\R\di \N$ in Definition \ref{standard} is also \emph{surjective}, as in e.g.\ \cite{hrbacekjech}.  
We refer to this stronger notion as `strongly countable', also studied in \cite{dagsamX}.

\smallskip

Now, a cursory search reveals that the word `countable' appears hundreds of times in the `bible' of RM \cite{simpson2}, and the same for \cite{dsliceke, simpson1}.  
Sections titles of \cite{simpson2} also reveal that the objects of study are `countable' rings, vector spaces, groups, et cetera.  
Of course, the above definition of `countable subset of $\R$' cannot be expressed in $\L_{2}$.  Indeed, all the 
aforementioned objects are given by \emph{sequences} in $\L_{2}$, which also constitutes the official definition of `countable set' as in \cite{simpson2}*{V.4.2}.  

\smallskip

In this light, the following `coding principle' $\cocode_{0}$ is \emph{crucial} to RM \emph{if} one wants the results in \cite{simpson2, simpson1, dsliceke} to have the same scope 
as third-order theorems about countable objects as in Definition \ref{standard}.  
This is particularly true for the RM of topology from \cites{mummy,mummyphd, mummymf}, as the latter is based on \emph{countable} bases at its very core.
\begin{princ}[$\cocode_{0}$]
For any non-empty countable set $A\subseteq [0,1]$, there is a sequence $(x_{n})_{n\in \N}$ in $A$ such that $(\forall x\in \R)(x\in A\asa (\exists n\in \N)(x_{n}=_{\R}x))$.
\end{princ}
As shown in \cite{samNEO}*{\S3}, $\cocode_{0}$ is equivalent to a number of natural principles, like the Bolzano-Weierstrass theorem for countable sets in $2^{\N}$.  
By \cite{dagsamX}*{\S3}, the latter theorem plus $\FIVE^{\omega}$ proves $\SIX$, while $\NIN$ does not have this property.  Note that one previously could only 
reach $\SIX$ via the RM of topology (\cites{mummy,mummyphd, mummymf}).  We note that $\FIVE^{\omega}$ is a conservative extension\footnote{The two final items of \cite{yamayamaharehare}*{Theorem 2.2} are (only) correct for $\QFAC$ replaced by $\QFAC^{0,1}$.} of $\FIVE$ for $\Pi_{3}^{1}$-formulas by \cite{yamayamaharehare}*{Theorem 2.2} and according to Rathjen in \cite{rathjenICM}*{\S3}, the strength of $\SIX$ \emph{dwarfs} that of $\FIVE$. 

\smallskip

In conclusion, while deriving $\NIN$ or $\NBI$ from a theorem of ordinary mathematics is interesting in its own right (as $\Z_{2}^{\omega}$ does not prove the former), 
deriving $\cocode_{0}$ puts the theorem in a completely different ballpark.  In contrast to \cite{dagsamX,samNEO}, some of the below theorems that imply $\cocode_{0}$ do not mention countable sets.   

\smallskip

Finally, we mention the `excluded middle trick' pioneered in \cite{dagsamV}.  As mentioned in Section \ref{HCT}, $(\exists^{2})$ is equivalent to the existence of a discontinuous function on $\R$ over $\RCAo$. 
Thus, $\neg(\exists^{2})$ is equivalent to \emph{all functions on $\R$ are continuous}, and the latter trivially implies $\NIN$ (and also $\cocode_{0}$).  
Hence, we have a proof of the latter in case of $\neg(\exists^{2})$.  If we now prove $X\di \NIN$ in $\ACAo$, the law of excluded middle as in $(\exists^{2})\vee \neg(\exists^{2})$ yields a proof of $X\di \NIN$ in $\RCAo$.
We will often make use of this trick when deriving $\NIN$ (or $\cocode_{0}$ or $\NBI$). 
\subsection{Baire classes}\label{afbaire}
In this section, we derive $\NIN$ from basic properties of Baire classes on the unit interval.

\smallskip

First of all, \emph{Baire classes} go back to Baire's 1899 dissertation (\cite{beren2}).
A function is `Baire class $0$' if it is continuous and `Baire class $n+1$' if it is the pointwise limit of Baire class $n$ functions.  
Each of these levels is non-trivial and there are functions that do not belong to any level, as shown by Lebesgue (see \cite{kleine}*{\S6.10}). 
Baire's \emph{characterisation theorem} (\cite{beren}*{p.\ 127}) expresses that a function is Baire class $1$ iff there is a point of continuity of the induced function on each perfect set.

\smallskip

Secondly, motivated by the previous, we consider the following principle $\BAIRE$ expressing that Baire class $2$ does not contain all functions.  We interpret `Baire class $1$' as the aforementioned equivalent condition involving perfect sets.  
The below principle $\BAIRE'$ deals with this equivalence. 
\begin{princ}[$\BAIRE$]
There is a function $f:[0,1]\di \R$ which is not the pointwise limit of Baire class $1$ functions on $[0,1]$. 
\end{princ}
The following proof still goes through if we require a modulus of convergence. 
\begin{thm}\label{CAS}
The system $\RCAo$ proves $\BAIRE\di \NIN$.
\end{thm}
\begin{proof}
Let $Y:[0,1]\di \N$ be an injection and let $f:[0,1]\di \R$ be any function.  Define the function $f_{n}:[0,1]\di \R$ as follows:
\[
f_{n}(x):=
\begin{cases}
f(x) & Y(x)\leq n \\
0 & \textup{otherwise}
\end{cases}.
\]
% $f(x)$ if $Y(x)\leq n$, and zero otherwise. 
Then $f_{n}$ is only (potentially) discontinuous at $n+1$ points, i.e.\ $f_{n}$ is of Baire class $1$ in the aforementioned formulation involving perfect sets.  Nonetheless, we clearly have $\lim_{n\di \infty}f_{n}=f$, even with a modulus as $f_{n}(x)=f(x)$ if $n\geq Y(x)$.
\end{proof}
Thirdly, the aforementioned characterisation theorem by Baire also gives rise to $\NIN$.  Indeed, the following theorem is the `sequential' version of the former.  
Recall we use the equivalent formulation of `Baire class 1' involving perfect sets. 
\begin{princ}[$\BAIRE'$]
Let $(f_{n})_{n\in \N}$ be a sequence of Baire class $1$ functions such that $\lim_{n\di\infty}f_{n}=f$.
Then there is a sequence $(g_{m, n})_{m, n\in \N}$ of continuous functions such that $\lim_{m\di \infty}g_{m, n}=f_{n}$ for each $n\in \N$ \(and $\lim_{n\di \infty}\lim_{m\di \infty}g_{m, n}=f$\).
\end{princ}
A proof of $\BAIRE'$ seems to require the Axiom of Choice to combine the approximations to $f_{n}$ into a sequence.  This explains the base theory in Theorem \ref{ACAC}. 
\begin{thm}\label{ACAC}
The system $\RCAo+\QFAC^{0,1}$ proves $\BAIRE'\di \NIN$.
\end{thm}
\begin{proof}
Let $Y:[0,1]\di \N$ be an injection and note that we have $(\exists^{2})$ by \cite{kohlenbach2}*{\S3}.
Consider the sequence $(f_{n})_{n\in \N}$ defined as follows:
\be\label{farflung}
f_{n}(x):=
\begin{cases}
1 & Y(x)\leq n \\
0 & \textup{otherwise}
\end{cases},
\ee
and note that $f_{n}:[0,1]\di \R$ has at most $n+1$ discontinuities, i.e.\ it is Baire class $1$ by the definition involving perfect sets.  
As in the proof of Theorem \ref{CAS}, we have $\lim_{n\di \infty}f_{n}=1$.  Hence, let $(g_{m,n})_{m,n\in \N}$ be the sequence of continuous functions provided by $\BAIRE'$.  
We have the following equivalences:
\begin{align*}
(\exists x\in [0,1])(Y(x)=n)
&\asa (\exists y\in [0,1])(f_{n}(y)=1)\\
&\textstyle \asa (\exists z\in [0,1])(\exists M\in \N)(\forall m\geq M)(g_{m, n}(z)>\frac{3}{4}  )\\
&\textstyle \asa (\exists q\in [0,1]\cap \Q)(\exists M\in \N)(\forall m\geq M)(g_{m, n}(q)>\frac{1}{2}  ),
\end{align*}
where the final equivalence follows by the continuity of $g_{m,n}$ for $m,n\in \N$.  
The final formula only involves quantifiers over $\N$ (up to coding) and using $(\exists^{2})$, there is $X\subset \N$ such that $n\in X\asa (\exists x\in [0,1])(Y(x)=n)$, i.e.\ $X$ is the range of $Y$.

\smallskip

Now apply $\QFAC^{0,1}$ to $(\forall n\in \N)(n\in X\di (\exists x\in [0,1])(Y(x)=n))$ and let $(x_{n})_{n\in \N}$ be the resulting sequence.  
Let $y\in [0,1]$ be such that $y\ne_{\R} x_{n}$ for all $n\in \N$, as provided by \cite{simpson2}*{II.4.9}.
Then $n_{0}:=Y(y)$ is such that $n_{0}\in X$ and hence $Y(x_{n_{0}})=n_{0}$ by definition, which is a contradiction as $y\ne_{\R} x_{n_{0}}$. 
\end{proof}
\noindent
Try as we might, the previous proofs (or principles) do not seem to yield $\cocode_{0}$.  

\smallskip

Finally, the previous results are interesting as $\Z_{2}^{\omega}+\QFAC^{0,1}$ cannot prove $\NIN$ by \cite{dagsamX}*{Theorem 3.1}.  
Hence, assuming $\Z_{2}$ is consistent, the stronger system $\Z_{2}^{\omega}+\QFAC^{0,1}+\neg \NIN$ is also consistent.  
Therefore, if one believes that (most) of ordinary mathematics can be developed in $\Z_{2}$ (see e.g.\ \cite{simpson2}*{p.\ xiv} for this claim), 
it is consistent with ordinary mathematics that there is nothing beyond Baire class $2$.  In our opinion, this shows that ordinary mathematics (beyond continuous functions)
cannot be developed in $\Z_{2}$ and that third-order arithmetic is needed.

%As it happens, the previous proof  generalises as follows. 
%\begin{cor}
%The system $\RCAo+\QFAC^{0,1}$ proves $\BAIRE'\di \cocode_{0}$.
%\end{cor}
%\begin{proof}
%In short, repeat the proof of the theorem for $[0,1]$ replaced by some countable set $A\subset [0,1]$, i.e.\ $Y:[0,1]\di \N$ is an injection on $A$.
%
%\smallskip
%
%In more detail, the first case in \eqref{farflung} should be changed to `$x\in A\wedge Y(x)=n$'.
%The resulting $f_{n}$ still has at most one discontinuity and $\lim_{n\di \infty}f_{n}$ is the characteristic function of $A$.
%Applying $\BAIRE'$, one similarly obtains 
%\begin{align*}
%(\exists x\in A)(Y(x)=n)
%&\asa (\exists y\in A)(f_{n}(y)=1)\\
%&\textstyle \asa (\exists z\in A)(\exists M\in \N)(\forall m\geq M)(g_{m, n}(z)>\frac{3}{4}  )\\
%&\textstyle \asa (\exists q\in [0,1]\cap \Q)(\exists M\in \N)(\forall m\geq M)(g_{m, n}(q)>\frac{1}{2}  ),
%\end{align*}
%
%
%\end{proof}
%We note that (the proof of) Theorem \ref{CAS} does not readily generalise in the same way as (the proof of) Theorem \ref{ACAC} does.  
%We do not have a conceptual explanation for this.  

\subsection{Unordered sums}\label{uncountsyn}
In this section, we consider \emph{unordered sums}, which are a device for bestowing meaning upon `uncountable sums' $\sum_{x\in I}f(x)$ for any index set $I$ and $f:I\di \R$.  
A central result is that if $\sum_{x\in I}f(x)$ somehow exists, it must be a `normal' series of the form $\sum_{i\in \N}f(y_{i})$; Tao mentions this theorem in \cite{taomes}*{p.~xii}. 
We show that basic versions of this theorem yield $\NIN$, $\cocode_{0}$, and $\NBI$.  As it turns out, the exact formulation of `$\sum_{x\in I}f(x)$ exists' makes a huge difference.  

\smallskip

First of all, by way of motivation, there is considerable historical and conceptual interest in this topic: Kelley notes in \cite{ooskelly}*{p.\ 64} that E.H.\ Moore's study of unordered sums in \cite{moorelimit2} led to the concept of \emph{net} with his student H.L.\ Smith (\cite{moorsmidje}).
Unordered sums can be found in (self-proclaimed) basic or applied textbooks (\cites{hunterapp,sohrab}) and can be used to develop measure theory (\cite{ooskelly}*{p.\ 79}).  
Moreover, Tukey shows in \cite{tukey1} that topology can be developed using \emph{phalanxes}, which are nets with the same index sets as unordered sums.  

\smallskip

Secondly, we have previously studied the RM of nets in \cites{samcie19, samnetspilot, samwollic19}, to which we refer for the definition of net in $\RCAo$.
Now, an unordered sum is just a special kind of \emph{net} and $a:[0,1]\di \R$ is therefore written $(a_{x})_{x\in [0,1]} $ to suggest the connection to nets.  
Let $\fin(\R)$ be the set of all finite sequences of reals without repetitions.  
\begin{defi}\label{kaukie2}\rm
We say that $(a_{x})_{x\in [0,1]}$ is \emph{summable} if $\lim_{F}x_{F}$ exists for the net $x_{F}:=\lambda F.\sum_{x\in F}a_{x}$ where $F$ ranges over $\fin(\R)$ and is ordered by inclusion.  
We write $\sum_{x\in [0,1]}a_{x}:=\lim_{F}x_{F}$ in case $(a_{x})_{x\in [0,1]}$ is summable. 
\end{defi}
Summability of course is equivalent to a version of the Cauchy criterion (see e.g.\ \cite{sohrab}*{p.\ 74} or \cite{hunterapp}*{p.\ 136}), as follows.   
\bdefi\label{kaukie}
We say that $(a_{x})_{x\in [0,1]} $ is \emph{Cauchy} if for $\eps>0$ there is $I\in \fin(\R)$ such that for all $J\in \fin({\R})$ with $J\cap I=\emptyset$, we have $|\sum_{x\in J}a_{x}|<\eps$.
\edefi
The following result, immediate from e.g.\ \cite{sohrab}*{Cor.\ 2.4.4} or \cite{hunterapp}*{p.\ 136}, expresses that if $\sum_{x\in [0,1]}a_{x}$ exists, it is actually just a `normal' series $\sum_{i\in \N}a_{y_{i}}$.
\begin{princ}[$\SUM$]
If $(a_{x})_{x\in [0,1]}$ is Cauchy and non-negative, then there is a sequence of reals $(y_{i})_{i\in \N}$ in $[0,1]$ such that $a_{y}=_{\R}0$ in case $(\forall i\in \N)(y\ne_{\R} y_{i})$.  % for all $i\in \N$.
\end{princ}
One readily proves $\SUM$ using $\BOOT+\QFAC^{0,1}$ based on \cite{samph}*{\S3.2}. 
Now let $\IND$ be the induction axiom for formulas in $\L_{\omega}$. 
\begin{thm}\label{franoc}
The system $\RCAo+\IND$ proves $\SUM \di \NIN$. 
\end{thm}
\begin{proof}
As noted in Section \ref{pintro}, we may assume $(\exists^{2})$.
We first note some technical results that pertain to induction.

\smallskip

First of all, it is well-known that the induction axiom yields `bounded comprehension'; for instance, $\Sigma_{1}^{0}$-induction implies that for $f:\N^{2}\di\N$ and $n\in \N$, 
there is a set $X\subset \N$ such that $m\in X\asa (\exists k\in \N)(f(m,k)=0)$ for any $m\leq n$ (\cite{simpson2}*{II.3.9}).
One similarly establishes that $\IND$ implies that for $Y:\R\di \N$ and $n\in \N$, there is $X\subset \N$ such that $m\in X\asa (\exists x\in \R)(Y(x)=m)$ for any $m\leq n$.
In this way, finite segments of the range of $Y:\R\di \N$ exists given $\IND$.  % while $\BOOT$ is necessary to obtain the entire range (see \cite{samph}*{\S3}).  

\smallskip

Secondly, the following is from the proof of \cite{dagsamV}*{Cor.\ 4.7}.  It is well-known that $\ZF$ proves the ÔfiniteÕ axiom of choice via mathematical induction (see e.g.\ \cite{tournedous}*{Ch.~IV}). Similarly, one uses $\IND$ to prove for $Z:(\N^{\N}\times \N)\di \N$ and $ n\in \N$:
\[
(\forall m\leq n)(\exists f\in \N^{\N})(Z(f,m)=0)\di (\exists w^{1^{*}})\big[|w|=n+1\wedge  (\forall m\leq n)(Z(w(i), m)=0)].
\]
One readily replaces variables over $\N^{\N}$ by variables over $\R$ or $[0,1]$.

\smallskip

Thirdly, fix an injection $Y:[0,1]\di \N$ and define $a_{x}:= \frac{1}{2^{Y(x)}}$.  To show that $(a_{x})_{x\in [0,1]}$ is Cauchy, fix $\eps>_{\R}0$ and let $n\in \N$ be such that $\frac{1}{2^{n}}<\eps$.
Let $X$ be such that $m\in X\asa (\exists x\in [0,1])(Y(x)=m)$ for any $m\leq n$ and define $Z(x, k)$ as $0$ in case $(k\not\in X)\vee [Y(x)=k \wedge k\in X]$, and $1$ otherwise.
Clearly, $(\forall m\leq n)(\exists x\in [0,1])(Z(x,m)=0)$ and let $w=\langle y_{0}, \dots, y_{n}\rangle$ be a finite sequence of reals of length $n+1$ such that $(\forall m\leq n)(Z(w(m), m)=0)$. 
Let $v$ be $w$ minus all $w(i)$ that do not satisfy $Y(w(i))=i$ for $i<|w|$.  Then we have for $m\leq n$:
\be\label{hug}
m\in X\asa (\exists x\in [0,1])(Y(x)=m)\asa (\exists j<|v|)(Y(v(j))=m).
\ee
Let $I$ be the finite set consisting of the reals in $v$.   Since $Y:[0,1]\di \N$ is an injection, we have  $|\sum_{x\in J}a_{x}|\leq \sum_{m=n+1}^{\infty}\frac{1}{2^{m}}= \frac{1}{2^{n}}<\eps$ for $J\in \fin(\R)$ such that $I\cap J=\emptyset$.  
Hence, $(a_{x})_{x\in [0,1]}$ is Cauchy and let $(y_{i})_{i\in \N}$ be as provided by $\SUM$.  Now find $z_{0}\in [0,1]$ not in this sequence using \cite{simpson2}*{II.4.9} and note that $a_{z_{0}}=\frac{1}{2^{Y(z_{0})}}=0$ yields a contradiction, and $\NIN$ follows. 
\end{proof}
%Note that in the previous theorem, we can restrict $\SUM$ to \emph{non-negative} $(a_{x})_{x\in [0,1]}$, i.e\ such that $a_{x}\geq 0$ for $x\in [0,1]$.
We could weaken $\SUM$ to the conclusion that for \emph{almost all} $y\in [0,1]$, $a_{y}=_{\R}0$ and the previous proof still goes through.
We could also restrict $\SUM$ to $a:[0,1]\di \R^{+}\cup\{0\}$ having uniformly bounded sums $\sum_{x\in F}a_{x}$ for $F\in \fin(\R)$ as in \cite{ooskelly}*{p.\ 78}.
The following corollary however needs the `full strength' of $\SUM$.
\begin{cor}
The system $\RCAo+\IND$ proves $\SUM\di \cocode_{0}$.  
\end{cor}
\begin{proof}
Fix a countable set $A\subset [0,1]$ and let $Y$ be injective on $A$.  
Define $a_{x}$ as $\frac{1}{2^{Y(x)}}$ if $x\in A$, and zero otherwise. 
Using \eqref{hug}, one similarly proves that $(a_{x})_{x\in [0,1]}$ is Cauchy, and let $(y_{i\in \N})$ be as provided by $\SUM$.  
If necessary, trim the sequence $(y_{i})_{i\in \N}$ using $\exists^{2}$ to make sure $a_{y_{i}}\ne_{\R} 0$ for all $i\in \N$. 
By definition, we have the following for all $x\in [0,1]$:
\[
x\in A\asa a_{x}\ne_{\R} 0\asa (\exists i\in \N)(x=_{\R}y_{i}), 
\]
which immediately yields $\cocode_{0}$, as required. 
\end{proof}
%Note that the proof of Theorem \ref{franoc} almost immediately generalises to a proof of the corollary. 
%
%\smallskip
Next, we study a version of $\SUM$ involving summability as in Definition \ref{kaukie2}.
\begin{princ}[$\SUM'$]
If $(a_{x})_{x\in [0,1]}$ is summable and non-negative, then there is a sequence of reals $(y_{i})_{i\in \N}$ in $[0,1]$ such that $a_{y}=_{\R}0$ in case $(\forall i\in \N)(y\ne_{\R} y_{i})$.  % for all $i\in \N$.
\end{princ}
To make sure (the net in) $\SUM'$ is well-defined, we shall always assume $(\exists^{2})$ to be given, an assumption we did not have to make for $\SUM$.  
\begin{thm}\label{quil}
The system $\ACAo+\IND$ proves $\SUM'\di \NBI$, while the system $\ACAo$ proves $\QFAC^{0,1}\di \SUM'$.
\end{thm}
\begin{proof}
For the first part, let $Y:[0,1]\di \N$ be a bijection and consider $a_{x}:= \frac{1}{2^{Y(x)}}$.  Using $\IND$ as in the proof of Theorem \ref{franoc}, one proves that $(a_{x})_{x\in [0,1]}$ is summable and $1=\sum_{x\in [0,1]}a_{x}$.  
Let $(y_{i})_{i\in \N}$ be as in $\SUM'$ and find $z_{0}\in [0,1]$ not in this sequence using \cite{simpson2}*{II.4.9}.  Then $a_{z_{0}}=\frac{1}{2^{Y(z_{0})}}=0$ yields a contradiction, and $\NBI$ follows as required. 

\smallskip

For the second part, apply $\QFAC^{0,1}$ to $(\forall k^{0})(\exists w^{1^{*}})(|b-\sum_{x\in w}a_{x}|<\frac{1}{2^{k}})$, where $b$ is the (net) limit $\sum_{x\in [0,1]}a_{x}$.
The resulting sequence $(w_{n})_{n\in \N}$ includes all the $x\in [0,1]$ such that $a_{x}\ne 0$, by the definition of net limit.  Indeed, suppose $z\in [0,1]$ is not in this sequence and $a_{z}>\frac{1}{2^{k_{0}}}>0$.
By the definition of net limit, $b$ is the supremum of $\sum_{x\in w}a_{x}$ for all $w^{1^{*}}$, but also $b<\sum_{x\in u}a_{x}$ where $u:=w_{k_{0}+1}*\langle z\rangle$, a contradiction, and $\SUM'$ follows. 
%there is $w^{1^{*}}$ such that for all $v\supseteq w$, we have $|b-\sum_{x\in v}a_{x}|<\eps/2$.  
\end{proof}
Let $\cocode_{1}$ be $\cocode_{0}$ restricted to \emph{strongly countable} sets as introduced right after Definition \ref{standard}.
Note that $\QFAC^{0,1}\di \cocode_{1}$ as shown in \cites{dagsamX, dagsamIX}.
\begin{cor}
The system $\ACAo+\IND$ proves $\SUM'\di \cocode_{1}$.  
\end{cor}
In conclusion, Theorems \ref{franoc} and  \ref{quil} are interesting as they show that the slight shift from the Cauchy condition in Definition \ref{kaukie} to summability as in Definition~\ref{kaukie2} yields a theorem that is much weaker.  Indeed $\Z_{2}^{\omega}+\QFAC^{0,1}$ cannot prove $\NIN$, while $\RCAo+\QFAC^{0,1}$ does prove $\NBI$, as shown in \cite{dagsamX}*{\S3}.
Moreover, $\SUM$ can be formulated without mentioning net limits, making it `more finitary'. 

\smallskip

Finally, we note that our results on nets from \cites{samnetspilot, samcie19, samwollic19} were the inspiration 
for some of the results in Section \ref{donk} on metric spaces, as follows. 

%If we only need to prove $\NBI$, then the sum of $a_{x}:= \frac{1}{2^{Y(x)}}$ converges to $1$.  
%Hence, we can use the usual definition (and not the Cauchy one) involving the limit. 
%Moreover, we can weaken the conclusion to `the set $\{x\in [0,1]: a_{x}\ne 0\}$ has measure zero', or even `measure $<1$'.

\subsection{Metric spaces}\label{donk}
\subsubsection{Introduction}\label{blintro}
In this section, we derive $\NIN$, $\NBI$, and $\cocode_{0}$ from various basic theorems pertaining to metric spaces, namely the following.
\begin{itemize}
\item The local equivalence between sequential and epsilon-delta continuity. 
\item The separability of certain metric spaces $(M, d)$ for $M\subseteq [0,1]$.
\end{itemize}
We first introduce some necessary definitions.  We emphasise that we only study metric spaces $(M, d)$ where $M$ is a subset of $\N^{\N}$ (modulo possible coding).  

\smallskip

We study metric spaces $(M, d)$ as in Definition~\ref{donc}, where $M$ comes with its own equivalence relation `$=_{M}$' and the metric $d$ satisfies 
the axiom of extensionality on $M$, i.e.\ $(\forall x, y, v, w\in M)\big([x=_{M}y\wedge v=_{M}w]\di d(x, v)=_{\R}d(y, w)\big)$.  
\bdefi\label{donc}
A functional $d: M^{2}\di \R$ is a \emph{metric on $M$} if it satisfies the following properties for $x, y, z\in M$:
\begin{enumerate}
 \renewcommand{\theenumi}{\alph{enumi}}
\item $d(x, y)=_{\R}0 \asa  x=_{M}y$,
\item $0\leq_{\R} d(x, y)=_{\R}d(y, x), $
\item $d(x, y)\leq_{\R} d(x, z)+ d(z, y)$.
\end{enumerate}
We use standard notation like $B_{d}^{M}(x, r)$ to denote $\{y\in M: d(x, y)<r\}$.
\edefi
To be absolutely clear, quantifying over $M$ amounts to quantifying over $\N^{\N}$ or $\R$, perhaps modulo coding, i.e.\ the previous definition can be made in third-order arithmetic for the intents and purposes of this paper.  

\smallskip

The following definitions are then standard.  
\bdefi[Countably-compact]
A metric space $(M, d)$ is \emph{countably-compact} if for any $(a_{n})_{n\in \N}$ in $M$ and sequence of rationals $(r_{n})_{n\in \N}$ such that we have $M\subset \cup_{n\in \N}B^{M}_{d}(a_{n}, r_{n})$, there is $m\in \N$ such that  $M\subset \cup_{n\leq m}B^{M}_{d}(a_{n}, r_{n})$.
\edefi
\bdefi[Separability]\label{SEPKE} 
A metric space $(M, d)$ is \emph{separable} if there is a sequence $(x_{n})_{n\in \N}$ in $M$ such that $(\forall x\in M, k\in \N)(\exists n\in \N)( d(x, x_{n})<\frac{1}{2^{k}})$.
\edefi
We note that Definition \ref{SEPKE} is used in constructive mathematics (see \cite{troeleke2}*{Ch.\ 7, Def.\ 2.2}).  
Our notion of separability is also implied by \emph{total boundedness} as used in RM (see \cite{simpson2}*{III.2.3} or \cite{browner}*{p.\ 53}).
According to Simpson (\cite{simpson2}*{p.\ 14}), one cannot speak at all about non-separable spaces in $\L_{2}$.

\subsubsection{Sequential continuity}
We show that $\NBI$ (and not $\NIN$) follows from the equivalence between sequential and `epsilon-delta' continuity in metric spaces. 

\smallskip

As to background,  as shown in \cite{kohlenbach2}*{Prop.\ 3.6}, $\RCAo+\QFAC^{0,1}$ is strong enough to show the \emph{local/pointwise} equivalence between sequential and epsilon-delta continuity on Baire space (or $\R$), while $\ZF$ cannot prove this equivalence, as noted in \cite{kohlenbach2}*{Remark 3.13}.  By \cite{heerlijk}*{Theorem 4.54}, working over $\ZF$, the axiom of countable choice for $\R$ is equivalent to the aforementioned local/pointwise equivalence for $\R$; the latter is therefore much weaker than the former over $\RCAo$, as $\NIN$ readily follows from countable choice for $\R$, but not from $\QFAC^{0,1}$ by \cite{dagsamX}*{Theorem 3.1}.

\smallskip

Since $\QFAC^{0,1}$ readily implies $\NBI$, it is a natural question whether a version of the aforementioned local/pointwise equivalence implies $\NBI$.  
We provide a positive answer, as follows.  Recall that the assumption on $M$ from Section \ref{blintro}.  % in the form of Theorem~\ref{klon}. 

\begin{thm}\label{klon}
The system $\RCAo$ proves that $\NBI$ follows from 
\begin{center}
\emph{For a metric space $(M, d)$, for any $x\in M$ and $F:M\di M$, if $F$ is sequentially continuous at $x$, then $F$ is continuous at $x$.}
\end{center}
The system $\ACAo+\QFAC^{0,1}$ proves the centred theorem. 
\end{thm}
\begin{proof}
As in the previous proofs, we may assume $(\exists^{2})$. 
For the first part, suppose $Y:[0,1]\di \N$ is a bijection. Define $M$ as the union of $\{0_{M}\}$ and the set $N:=\{ w^{1^{*}}: (\forall i<|w|)(Y(w(i))=i)\}$. 
This definition makes sense because $Y$ is a bijection. 
We define `$=_{M}$' as $0_{M}=_{M}0_{M}$, $u\ne_{M} 0_{M}$ for $u\in N$, and $w=_{M}v$ if $w=_{1^{*}}v$ and $w, v\in N$.
The metric $d:M^{2}\di \R$ is defined as $d(0_{M}, 0_{M})=_{\R}0$, $d(0_{M}, u)=d(u, 0_{M})=\frac{1}{2^{|u|}}$ for $u\in N$ and $d(w, v)=|\frac{1}{2^{|v|}}-\frac{1}{2^{|w|}}|$ for $w, v\in N$.
Since $Y$ is an injection, we have $d(v, w)=_{\R}0 \asa  v=_{M}w$.  The other properties of a metric space from Definition \ref{donc} follow by definition.  

\smallskip
  
Now define the function $F: M\di M$ as follows: $F(0_{D}):= 0_{D}$ and $F(w):=u_{0}$ for any $w\in N$ and some fixed $u_{0}\in N$.   
Clearly, if the sequence $(w_{n})_{n\in \N}$ converges to $0_{D}$, either it is eventually constant $0_{D}$ or lists all reals in $[0,1]$.  The latter case is impossible by Cantor's theorem (\cite{simpson2}*{II.4.9}).  
Hence, $F$ is sequentially continuous at $0_{D}$, but clearly not continuous at $0_{D}$.  This contradiction yields $\NBI$.

\smallskip

For the second part, suppose $F:M\di M$ is not continuous at $y_{0}\in M$, i.e.\
\[\textstyle
(\exists k_{0}\in \N)(\forall N\in \N)(\exists y\in M)(d(y_{0},y)<\frac{1}{2^{N}}\wedge d(F(y), F(y_{0}))>\frac{1}{2^{k_{0}}}).
\]
Fix such $k_{0}$ and apply $\QFAC^{0,1}$ (using $\exists^{2}$) to obtain a sequence $(x_{n})_{n\in \N}$ in $M$ such that $y_{0}=_{M}\lim_{n\di \infty}x_{n}$ but $F(y_{0})\ne_{M} \lim_{n\di \infty}F(x_{n})$.
\end{proof}
A function $F$ is \emph{net continuous} at $x$ if for any net $(x_{d})_{d\in D}$ converging to $x$, the net $(F(x_{d}))_{d\in D}$ converges to $F(x)$.
The RM-study of this notion is in \cite{samnetspilot}*{\S4}.
\begin{cor}
The first part of the theorem remains valid if we replace `$F$ is continuous at $x$' by `$F$ is net continuous at $x$'.
\end{cor}
\begin{proof}
Consider $M$ and $F$ from the proof of the theorem.
Define a net $(x_{d})_{d\in D}$ by $D=N$ and $x_{d}:=d$, and order this set by inclusion.  
Clearly, this net converges to $0_{D}$, while $F(0_{D})$ is different from the limit of the net $(F(x_{d}))_{d\in D}$.
\end{proof}
On a conceptual note, it is well-known that topologies cannot always be described in terms of sequences, but nets are needed instead.  
The previous proof provides a nice example of a space $M$ (which exists in $\Z_{2}^{\omega}+\neg\NBI$) in which no non-trivial \emph{sequence} converges to $0_{D}$, while 
there is a non-trivial net that converges to $0_{D}$.  

\smallskip

Inspired by the previous, recall that a space is called \emph{sequential} if the usual definition of closed set (complement of an open set) coincides with the sequential definition (closed under limits of sequences). 
Note that for $M$ as in the proof of the theorem, the set $N$ is sequentially closed, since a sequence that converges to $0_{D}$ must be eventually constant $0_{D}$.  
Clearly, $N$ is not closed as $\{0_{D}\}$ is not open.  Hence, $\NBI$ follows from the statement that \emph{any metric space is sequential}. 

\subsubsection{Separability of metric space}
In this section, we show that $\cocode_{0}$ is implied by a basic `separability' theorem called $\STS^{+}$, which does not mention the notion of `countable set' as in Definition \ref{standard}.  
The results are based on \cite{dagsamX}*{\S3.1.3}, where $\NIN$ is derived from $\STS^{+}$ restricted to $A=[0,1]$. 
% In particular, we show that a slight generalisation of $\STS$, called $\STS^{+}$ below, yields $\cocode_{0}$.  
%This is cool because $\STS^{+}$ and variations do not mention `countable set' as in Definition \ref{standard}.
%\smallskip
%
%In this section, we derive $\NIN$ from a most basic separability property of metrics on the unit interval. 

\smallskip

First of all, as to motivation, the study of metric spaces in RM proceeds -unsurprisingly- via codes, namely a complete separable metric space is represented via a countable and dense subset (\cite{simpson2}*{II.5.1}). 
It is then a natural question how hard it is to prove that this countable and dense subset exists for the original/non-coded metric spaces. 
We study the special case for metrics \emph{defined on sub-sets of the unit interval}, as in Definition \ref{donc} and $\STS^{+}$ below, which implies $\cocode_{0}$ by Theorem \ref{STS}.  
Our interest in $\STS^{+}$ lies with Corollary \ref{BOOM}.
%The following definitions are quite standard.

\begin{princ}[$\STS^{+}$]
For any $M\subseteq [0,1]$, if $(M, d)$ is 
a countably-compact metric space, then it is separable. 
\end{princ}
\begin{thm}\label{STS}
The system $\RCAo$ proves $\STS^{+}\di \cocode_{0}$.
\end{thm}
\begin{proof}
Recall that by \cite{kohlenbach2}*{\S3}, $\cocode_{0}$ trivially holds if $\neg(\exists^{2})$ as in the latter case all functions on $\R$ are continuous.  
Thus, we may assume $(\exists^{2})$ for the rest of the proof.  Let $A\subset [0,1]$ be a countable set and let $Y:[0,1]\di \N$ be injective on $A$.
Without loss of generality, we may assume that $0\in A$.

\smallskip

Define $d(x, y):=|\frac{1}{2^{Y(x)}}-\frac{1}{2^{Y(y)}}|$ in case $x, y\in A$ are non-zero.  Define $d(0, 0):=0$ and $d(x, 0)= d(0, x):= \frac{1}{2^{Y(x)}}$ for non-zero $x\in A$.
The first item in Definition~\ref{donc} holds by the assumption on $Y$, while the other two items hold by definition. 

\smallskip

The metric space $(A, d)$ is countably-compact as $0\in B_{d}(x, r)$ implies $y\in B_{d}(x, r)$ for $y\in A$ with only finitely many exceptions (due to $Y$ being an injection).  
Let $(x_{n})_{n\in\N}$ be the sequence provided by $\STS^{+}$, implying $(\forall x\in A)(\exists n\in \N)( d(x, x_{n})<\frac{1}{2^{Y(x)+1}})$ by taking $k=Y(x)+1$.  
The latter formula implies 
\be\label{jilogz}\textstyle
(\forall x\in A)(\exists n\in \N)(x\ne_{\R} 0\di  |\frac{1}{2^{Y(x)}}-\frac{1}{2^{Y(x_{n})}}|<_{\R}\frac{1}{2^{Y(x)+1}}) 
\ee
by definition.  Note that $x_{n}$ from \eqref{jilogz} cannot be $0$ by the definition of the metric $d$.  
Clearly, $|\frac{1}{2^{Y(x)}}-\frac{1}{2^{Y(x_{n})}}|<\frac{1}{2^{Y(x)+1}}$ is only possible if $Y(x)=Y(x_{n})$, implying $x=_{\R}x_{n}$. 
Hence, we have shown that $(x_{n})_{n\in \N}$ lists all reals in $A\setminus\{0\}$.  %as required for $\cocode_{0}$.
%By \cite{simpson2}*{II.4.9}, there is $y\in [0,1]$ such that $y\ne x_{n}$ for all $n\in \N$ (in $\RCA_{0}$).  
%This contradiction implies $\NIN$. 
\end{proof}
By the previous proof, we may restrict $\STS^{+}$ to countable metric spaces, whence it becomes an extension of \cite{hirstrm2001}*{Theorem 1.item 2}.
The following corollary follows in the same way as \cite{dagsamX}*{Theorem 3.23}. 
\begin{cor}\label{BOOM}
The system $\FIVE^{\omega}+\STS^{+}$ proves $\SIX$.
\end{cor}
Finally, if we replace `separability' in $\STS^{+}$ by e.g.\ covering properties (Heine-Bore, Vitali, Lindel\"of), the resulting principle
does imply $\NIN$, namely by \cite{dagsamX}*{Cor.~3.11}, but no longer $\cocode_{0}$, it seems.

\begin{ack}\rm
We thank Anil Nerode for his helpful suggestions and Jeff Hirst and Carl Mummert for suggesting the principle $\NBI$ to us. 
Our research was supported by the John Templeton Foundation via the grant \emph{a new dawn of intuitionism} with ID 60842 and by the \emph{Deutsche Forschungsgemeinschaft} via the DFG grant SA3418/1-1.
Opinions expressed in this paper do not necessarily reflect those of the John Templeton Foundation.   
\end{ack}

%The following corollary is now fairly straightforward.  
%\begin{cor}
%The theorem still goes through upon replacing `separable' in $\STS^{+}$ by any of the following.
%
%\begin{enumerate}
% \renewcommand{\theenumi}{\alph{enumi}}
%\item Total boundedness as in \cite{simpson2}*{III.2.3} or \cite{browner}*{p.\ 53}.   
%\item The Heine-Borel property for \emph{uncountable} covers.\label{lappel}
%\item The Lindel\"of property for \emph{uncountable} covers.\label{lappel2}
%\item The Vitali covering property as in $\WHBU$ for \emph{uncountable} Vitali covers.  \label{frap}
%\end{enumerate}
%\end{cor}
%\begin{proof}
%For item \eqref{lappel}, fix $A\subset [0,1]$, let $Y:[0,1]\di \N$ be an injection on $A$, and let $d$ be the metric as in the proof of the theorem.  
%Then $B_{d}(x, r)=\{x\}$ for $r>0$ small enough and $x\in A\setminus \{0\}$.  
%In particular, the covering $\cup_{x\in A}B_{d}(x, \frac{1}{2^{Y(x)+1}})$ of $A$ cannot have a finite (or countable) sub-cover, and $\NIN$ follows. 
%
%\smallskip
%
%For item \eqref{frap} (and item \eqref{lappel2}), let $Y$ and $d$ be as in the previous paragraph and note that $\cup_{x\in [0,1]}B_{d}(x, \frac{1}{2^{Y(x)+1}})$ is a Vitali cover.  
%By the above, no finite sum can be larger than $1/2$, and we are done. 
%\end{proof}
%It should also be straightforward to derive $\NIN$ from the \emph{non}-separability of e.g.\ the sequence space $\ell^{\infty}$ or the space $\textsf{BV}$ of functions of bounded variation.  
%

\bye